\documentclass[10pt]{amsart}
\usepackage[top=3.5cm,bottom=3.5cm,left=4cm,right=4cm]{geometry}               
\usepackage{amsmath,amscd,amssymb,amsthm}
\usepackage[english]{babel}
\usepackage{mathtools}
\usepackage{enumerate}
\usepackage{setspace}
\usepackage{mathrsfs}
\usepackage{xcolor}
\usepackage[numbers]{natbib}
\usepackage[hidelinks]{hyperref}

\DeclareMathOperator{\Div}{Div}

\def\vF{\mathbb{F}}

\def\vZ{\mathbb{Z}}

\def\vP{\mathbb{P}}

\def\cN{\mathcal{N}}
\def\cO{\mathcal{O}}
\def\cS{\mathcal{S}}

\def\cD{\mathcal{D}}
\def\cE{\mathcal{E}}

\newtheorem{theorem}{Theorem}
\newtheorem{lemma}[theorem]{Lemma}

\newtheorem{proposition}[theorem]{Proposition}
\theoremstyle{definition}
\newtheorem{definition}[theorem]{Definition}

\makeindex
\thanks{The second author was partially supported by Swiss National Science Foundation grant number 149716 and \emph{Armasuisse}}

 \author[E. Dotti]{Edoardo Dotti}
 \address{Institute of Mathematics\\ 
 University of Zurich\\
 Winterthurerstrasse 190\\
 8057 Zurich, Switzerland\\
 }
 \email{edoardo.dotti@uzh.ch}

 \author[G. Micheli]{Giacomo Micheli}
 \address{Institute of Mathematics\\ 
 University of Zurich\\
 Winterthurerstrasse 190\\
 8057 Zurich, Switzerland\\
 }
 \email{giacomo.micheli@math.uzh.ch}

\date{}
\title{Eisenstein polynomials over Function fields}
\begin{document}
\begin{abstract}
In this paper we compute the density of monic and non-monic Eisenstein polynomials of fixed degree having entries in an integrally closed subring of a function field over a finite field.
\end{abstract}
\maketitle
\smallskip
\noindent \textbf{Keywords}: Function fields, Density, Polynomials, Riemann-Roch spaces. \\
\smallskip
\noindent \textbf{MSC}: 11R58, 11T06 
\section{Introduction}
%

Let us start with the definition of \emph{Eisenstein polynomial} and \emph{natural density}
\begin{definition}
Let $R$ be an integral domain.
A polynomial $f(X)=\sum^{n}_{i=0} a_i x^i \in R[X]$ is said to be Eisenstein if there exists a prime ideal $\mathfrak{p}\subseteq R$ for which 
\begin{itemize}
\item $a_i\in \mathfrak{p}$ for all $i\in \{0,\dots,n-1\}$,
\item $a_0\notin \mathfrak{p}^2$, 
\item $a_n\notin \mathfrak{p}$.
\end{itemize}
\end{definition}
\begin{definition}
A subset $A$ of $\vZ^n$ is said to have \emph{density} $a$ if \[a=\lim_{B\rightarrow \infty} \frac{|A\cap [-B,B[^n|}{(2B)^n}.\]
\end{definition}
A classical result from the literature is that any Eisenstein polynomial is irreducible. In addition, observe that any polynomial of degree at most $d$ and coefficients over $\vZ$ can be regarded as an element of $\vZ^{d+1}$, while any monic polynomial of degree $d$ can be regarded as an element of $\vZ^d$.
Recently, it has been of interest the explicit computation of the natural density of both degree $d$ Eisenstein polynomials and monic Eisenstein polynomials over $\vZ$, see for example \cite{heyman2013number,dubickas2003polynomials}.


As was first proved by Dubickas in \cite{dubickas2003polynomials}, the natural density of monic Eisenstein polynomials over $\vZ$ of fixed degree $d$ is
\begin{equation}\label{eq:rational_eis}
\prod_{p\:\text{prime}} \left(1- \frac{p-1}{p^{d+1}}\right).
\end{equation}
Heyman and Shparlinski extended the results of Dubickas to general Eisenstein polynomials and computed the error term of the density \cite[Theorem 1, Theorem 2]{heyman2013number}.

In this paper we would like to establish a function field  analogue of these results that  will include all the cases in which $R$ is selected as an integrally closed subring of a function field of a curve over a finite field.

The general case that we will analyse needs an appropriate definition of density which makes use of Moore-Smith convergence for directed sets, as described in \cite[]{micheli2014density}. For the moment, let us fix the notation for the basic structures we are are going to deal with, which is essentially the same as in \citep{bib:stichtenoth2009algebraic}. 

Let $q$ be a prime power and $\vF_q$ be the finite field of order $q$. 
Let $F$ be a function field having full constant field $\vF_q$.
Let $\vP_F$ be the set of places of $F$ and $\cS$ a non empty proper subset of $\vP_F$. Let us denote by $\cO_P$ the valuation ring at a place $P$ of $F$.
Let $H=\bigcap_{P\in \cS} \cO_P$ be the holomorphy ring associated to $\cS$ \citep[Definition 3.2.2]{bib:stichtenoth2009algebraic}. As it is well known, $H$ is a Dedekind Domain therefore any prime ideal is also maximal. In addition the maximal ideals of $H$ correspond exactly to the places in $\cS$  see \cite[Proposition 3.2.9]{bib:stichtenoth2009algebraic}. Therefore, if $P$ is a place of $F$ which lies in $\cS$ there exists a unique maximal ideal $P_H\subseteq H$ corresponding to $P$ for which $P\cap H=P_H$. In order not to heavier the notation, we will denote $P_H$ again by $P$.
Let $\mathcal{D}$ be the set of positive divisors of $\Div(F)$ having support outside the holomorphy set $\mathcal S$. It is easy to observe
\[H=\bigcup_{D\in\cD}\mathcal{L}(D)\]
and that $\cD$ is a directed set.

Let now $ A\subseteq H^m$, we define the \emph{upper} and \emph{lower density} of $A$ as
\[\overline {\mathbb D}(A)=
\limsup_{D\in \mathcal D} \frac{|A\cap \mathcal L(D)^m|}{q^{m\ell(D)}},\]
\[\underline {\mathbb D}(A)=
\liminf_{D\in \mathcal D} \frac{| A\cap \mathcal L(D)^m|}{q^{m\ell(D)}}\]
where the limit is defined using Moore-Smith convergence over the directed set $\mathcal{\cD}$ (see \citep[Chapter 2]{bib:kelley1955general}).
The \emph{density} of $ A$ is then defined if $\underline {\mathbb D}({A})=\overline {\mathbb D}({A})=:\mathbb D({A})$.

As already observed in \cite[]{micheli2014density}  , if we specialize our definition of density to the case of the univariate polynomial ring over a finite field we get the usual definition of density for $\mathbb{F}_q[x]$, see for example in \cite{bib:guo2013probability, bib:sugita2007probability}.

In addition, the final formulas for the density we get are analogous to the ones over the rational integers obtained in \cite{dubickas2003polynomials,heyman2013number}. 

The paper is structured as follows: in the next subsection we specify the notation we are going to use for the rest of the paper, in section \ref{monic} we compute the density of monic Eisenstein polynomials, in section \ref{nonmonic} we apply a similar strategy to compute the density of general Eisenstein polynomials.

\subsection{Notation}

Throughout this paper,
when $Y$ is a set and $m$ is a positive integer, we will denote by $Y^m$ the cartesian product of $m$-copies of $Y$.
To avoid confusion, the square  of an ideal $Q$ will then be denoted by $\widehat{Q}=Q\cdot Q$.  Furthermore notice that in the whole paper we consider polynomials of degree $d>1$. To easier the notation, we fix an enumeration $\{Q_1,Q_2,\dots,Q_i,\dots\}$  of the places of $\cS$. 
Since we will deal with the density of both monic and non-monic Eisenstein polynomials, we have to distinguish the notation, which we clarify in the following two paragraphs. 

\subsubsection*{Notation for monic Eisenstein polynomials:}  with a small abuse of notation we identify $H^d$ with the set of all monic polynomials of degree $d$ having entries over $H$. In particular, if $(h_0,\dots,h_{d-1})\in H^d$ then $h_i$ denotes the coefficient of the monomial of degree $i$. 
Furthermore, we denote by $\cE\subset H^d$ the set of monic Eisenstein polynomials of degree $d$ and by $\cN$ its complement in $H^d$. 
We denote by $\cE_i$ the set of monic polynomials which are Eisenstein with respect to $Q_i$:
\[\cE_i=\{(h_0,\dots, h_{d-1})\in H^d: \quad h_i\in Q_i \; \forall i\in \{0,\dots d-1\}\;\text{and}\; h_0\notin \widehat Q_i \}.\]
We denote by $\cN_i$ the complement of $\cE_i$.

\subsubsection*{Notation for Eisenstein polynomials:}  Analogously, we identify the set of all polynomials of degree $d$ having entries over $H$ with $H^{d+1}$. 
Let  $\cE^+\subseteq H^{d+1}$ be the set of Eisenstein polynomials of degree $d$ and $\cN^+$ be its complement in $H^{d+1}$.
We denote by $\cE^+_i$ the set of polynomials which are  Eisenstein with respect to $Q_i$:
\[\cE^+_i=\{(h_0,\dots, h_{d})\in H^{d+1}: \quad h_i\in Q_i \; \forall i\in \{0,\dots d-1\},\; h_0\notin \widehat Q_i \;\text{and}\; h_d\notin Q_i\}.\]
We denote by $\cN^+_i$ the complement of $\cE^+_i$.

\section{Monic Eisenstein Polynomials}\label{monic}

In this section we compute the density of monic Eisenstein polynomials via approximating the complement of $\mathcal{E}$ (i.e. $\mathcal{N}$) with $\overline{\cN_t}=\bigcap_{i=1}^t \cN_i$. 
First we show that we can explicitly compute the density of $\overline{\cN_t}$  (Proposition \ref{approx}). Then, we give a criterion to check whether the approximation is ``sharp'': i.e. whether the limit of the densities of $\overline{\cN_t}$ converges to the density of $\cN$ (Lemma \ref{magic_lemma}). Finally, we verify that the conditions under which the approximation is sharp are verified (Theorem \ref{maintheorem}).

\begin{proposition}\label{approx} The density of $\overline{\cN_t}$ is
 \[\mathbb{D}(\overline{\cN_t})=\prod_{i=1}^t \left(1-\frac{q^{\deg(Q_i)}-1}{q^{(d+1)\deg(Q_i)}}\right).\]
\end{proposition}

\begin{proof}
Consider the map \[\tilde{\phi}:H^d\rightarrow \left(H / ( \widehat Q_1\cdots \widehat Q_t)\right)^d,\] which is defined componentwise by the reduction modulo the ideal $(\widehat Q_1\cdots \widehat Q_t)$. Observe also that $\left(H / ( \widehat Q_1\cdots \widehat Q_t)\right)^d\simeq\prod_{i=1}^t \left(H / \widehat Q_i\right)^d $ by the Chinese Remainder Theorem.\\
Consider now a divisor $D\in\mathcal{D}$. In order to compute the density of $\overline{\cN_t}$ it is enough to count how many elements there are in $\overline{\cN_t}\cap\mathcal{L}(D)^d$, when the degree of $D$ is large.

We start by showing that $\mathcal{L}(D)^d$ maps surjectively onto $\left(H / ( \widehat Q_1\cdots \widehat Q_t)\right)^d$ when the degree of $D$ is large enough.\\
For this consider the $\mathbb{F}_q$ linear map $\phi :\mathcal{L}(D)\rightarrow \left(H / ( \widehat Q_1\cdots \widehat Q_t)\right)$. We have $\ker(\phi)=\mathcal{L}(D)\cap  ( \widehat Q_1\cdots \widehat Q_t)$, which represents the elements of $\mathcal{L}(D)$ having at least a double root at each $Q_i$. Hence $\ker(\phi)=\mathcal{L}(D-2\sum_{i=1}^t Q_i)$. \\
By Riemann's theorem \citep[Theorem 1.4.17]{bib:stichtenoth2009algebraic}, if the degree of $D$ is large enough, the dimension of the kernel as an $\mathbb{F}_q$ vector space is
\begin{equation}\label{stralusch}
  \ell\left( D-2\sum_{i=1}^t Q_i\right)=\deg\left(D-2\sum_{i=1}^t Q_i\right)+1-g=\deg(D)-2\sum_{i=1}^t\deg(Q_i)+1-g,
\end{equation} 
  where $g$ denotes the genus of the function field.\\
By the same theorem $\ell(D)=\deg(D)+1-g$. Hence we obtain \[\dim_{\mathbb{F}_q}\left(\mathcal{L}(D) / \ker(\phi)\right)=\ell(D)-\ell\left(D-2\sum_{i=1}^t Q_i \right)=2\sum_{i=1}^t\deg(Q_i).\]
On the other hand, by the Chinese Remainder Theorem \[\dim_{\mathbb{F}_q}\left(H / ( \widehat Q_1\cdots \widehat Q_t)\right) \overset{CRT}{=} \dim_{\mathbb{F}_q}\left(H/\widehat Q_1\times\cdots\times H/\widehat Q_t \right)=2\sum_{i=1}^t\deg(Q_i) .\] Therefore when the degree of $D$ is large enough $\phi$ is surjective, thus $\tilde{\phi}$ is surjective.

Let $\psi_i:\left(H / (\widehat Q_1\cdots \widehat Q_t)\right)^d\longrightarrow \left(H / \widehat Q_i\right)^d$.
We have the following situation: \[ \mathcal{L}(D)^d\overset{\tilde{\phi}}{\twoheadrightarrow}\left(H / ( \widehat Q_1\cdots\widehat Q_t)\right)^d\overset{\psi}{\rightarrow}\prod_{i=1}^t \left(H / \widehat Q_i\right)^d,\] where $\psi =(\psi_1,\dots ,\psi_t)$. Notice that the check for $f\in H^d$ not to be Eisenstein with respect to $Q_i$ can be performed by looking at the reduction modulo $\widehat Q_i$. Therefore $f\in \overline{\cN_t}\cap\mathcal{L}(D)^d $ if and only if $\psi_i\circ\tilde{\phi}(f)\notin \left( (Q_i/\widehat Q_i)\setminus\{0\}\right)\times\left( Q_i/\widehat Q_i \right)^{d-1}=:E_i$ for all $i\in \{1,\dots ,t\}$.

It follows that  $\overline{\cN_t}\cap\mathcal{L}(D)^d=\tilde{\phi}^{-1}\left(\psi^{-1}\left(\prod_{i=1}^t\left( (H/\widehat Q_i)^d\setminus E_i\right) \right) \right)\cap\mathcal{L}(D)^d$. Hence
\begin{align*}
 \ \mid\overline{\cN_t}\cap\mathcal{L}(D)^d\mid &= \mid\ker(\tilde{\phi})\mid\cdot\prod_{i=1}^t\mid\left(H/\widehat Q_i\right)^d\setminus E_i\mid\\ &=q^{d\left(\deg(D)-2\sum_{i=1}^t\deg(Q_i)+1-g\right)}\cdot\prod_{i=1}^t\mid\left(H/\widehat Q_i\right)^d\setminus E_i\mid,
\end{align*}   
    where the last equality follows from (\ref{stralusch}). Now it remains to compute
\begin{align*}    \mid\left(H/\widehat Q_i\right)^d\setminus E_i\mid &=q^{2d\deg(Q_i)}-\mid\left( (Q_i/\widehat Q_i)\setminus\{0\}\right)\times\left( Q_i/\widehat Q_i \right)^{d-1}\mid\\ &=q^{2d\deg(Q_i)}-\left(q^{\deg(Q_i)}-1\right)\cdot q^{(d-1)\deg(Q_i)}\\&=q^{2d\deg(Q_i)}\left(1-q^{-d\deg(Q_i)}+q^{-(d+1)\deg(Q_i)} \right).
\end{align*}

Therefore for $D$ of degree large enough \[ \frac{\mid \overline{\cN_t}\cap\mathcal{L}(D)^d\mid}{\mid\mathcal{L}(D)^d\mid}=\frac{q^{d\left(\deg(D)-2\sum_{i=1}^t\deg(Q_i)+1-g\right)}}{q^{d\left(\deg(D)+1-g\right)}}\cdot\prod_{i=1}^t q^{2d\deg(Q_i)}\left(1-q^{-d\deg(Q_i)}+q^{-(d+1)\deg(Q_i)} \right)\] \[=\prod_{i=1}^t \left(1-q^{-d\deg(Q_i)}+q^{-(d+1)\deg(Q_i)} \right)=\prod_{i=1}^t \left(1-\frac{q^{\deg(Q_i)}-1}{q^{(d+1)\deg(Q_i)}}\right). \]
Hence \[\mathbb{D}(\overline{\cN_t})=\underset{D\in\mathcal{D}}{\lim}\frac{\mid \overline{\cN_t}\cap\mathcal{L}(D)^d\mid}{\mid\mathcal{L}(D)^d\mid}=\prod_{i=1}^t \left(1-\frac{q^{\deg(Q_i)}-1}{q^{(d+1)\deg(Q_i)}}\right).\]
\end{proof}
\begin{lemma}\label{magic_lemma}
Let $n\in\mathbb{N}$, $A\subseteq H^n$. Let $\{A_t\}_{t\in\mathbb{N}}$ be a family of subsets of $H^n$ such that $A_{t+1}\subseteq A_t$ and $\bigcap_{t\in\mathbb{N}} A_t=A$. Assume also that $\mathbb{D}(A_t)$ exists for all $t$. If $\lim_{t\rightarrow\infty}\overline{\mathbb{D}}(A_t\setminus A)=0$, then $ \mathbb{D}(A)=\lim_{t\rightarrow\infty}\mathbb{D}(A_t)$.
\end{lemma}
\begin{proof}
We start from the equality $\mid A_t\cap\mathcal{L}(D)^n\mid=\mid A\cap\mathcal{L}(D)^n\mid+\mid (A_t\setminus A)\cap\mathcal{L}(D)^n\mid$, from which it follows

\begin{align*}
 \underset{D\in\mathcal{D}}{\liminf}\frac{\mid A\cap\mathcal{L}(D)^n\mid}{\mid\mathcal{L}(D)^n\mid} &=\underset{D\in\mathcal{D}}{\liminf} \left( \frac{\mid A_t\cap\mathcal{L}(D)^n\mid}{\mid\mathcal{L}(D)^n\mid}-\frac{\mid (A_t\setminus A)\cap\mathcal{L}(D)^n\mid}{\mid\mathcal{L}(D)^n\mid} \right)\\&\geq \underset{D\in\mathcal{D}}{\liminf} \frac{\mid A_t\cap\mathcal{L}(D)^n\mid}{\mid\mathcal{L}(D)^n\mid}-\underset{D\in\mathcal{D}}{\limsup} \frac{\mid (A_t\setminus A)\cap\mathcal{L}(D)^n\mid}{\mid\mathcal{L}(D)^n\mid}.
\end{align*}

It follows that $\underline{\mathbb{D}}(A_t)-\overline{\mathbb{D}}(A_t\setminus A)\leq \underline{\mathbb{D}}(A)$. Since ${\mathbb{D}}(A_t)$ exists for all $t$ we get \[\mathbb{D}(A_t)-\overline{\mathbb{D}}(A_t\setminus A)\leq \underline{\mathbb{D}}(A).\]

Now notice that $\underset{t\rightarrow\infty}{\lim}\mathbb{D}(A_t)$ exists since $\mathbb{D}(A_t)$ is decreasing and bounded from below. By taking the limit in $t$, the last expression then becomes  \[\underset{t\rightarrow\infty}{\lim}\mathbb{D}(A_t)-\underset{t\rightarrow\infty}{\lim}\overline{\mathbb{D}}(A_t\setminus A)\leq \underline{\mathbb{D}}(A).\] 
Since $\underset{t\rightarrow\infty}{\lim}\overline{\mathbb{D}}(A_t\setminus A)=0$ by assumption, it follows that $\underset{t\rightarrow\infty}{\lim}\mathbb{D}(A_t)\leq \underline{\mathbb{D}}(A)$.

On the other hand $A\subseteq A_t$ which implies $\overline{\mathbb{D}}(A)\leq\mathbb{D}(A_t)$. In particular $\overline{\mathbb{D}}(A)\leq\underset{t\rightarrow\infty}{\lim}\mathbb{D}(A_t)$.\
Combining all together we get \[\underset{t\rightarrow\infty}{\lim}\mathbb{D}(A_t)\leq\underline{\mathbb{D}}(A)\leq\overline{\mathbb{D}}(A)\leq \underset{t\rightarrow\infty}{\lim}\mathbb{D}(A_t)  ,\]
therefore the claim follows.
\end{proof}

\begin{theorem}\label{maintheorem}
The density of the set of monic Eisenstein polynomials with coefficients in $H$ is 
\[\mathbb{D}(\cE)=1-\prod_{Q\in \cS}\left(1-\frac{q^{\deg(Q)}-1}{q^{(d+1)\deg(Q)}}\right).\]
\end{theorem}

\begin{proof}

We make use of Lemma \ref{magic_lemma} for the family $\{\overline{\cN_t}\}_{t\in\mathbb{N}}$. Hence we want to show that $ \lim_{t\rightarrow\infty}\overline{\mathbb{D}}(\overline{\cN_t}\setminus \cN)=0$.\\
First note that
\begin{itemize}
\item $\overline{\cN_t}\setminus\cN =\bigcup_{r>t}\cE_r\subseteq \bigcup_{r>t} Q_r^d$ ,
\item $Q_r^d\cap\mathcal{L}(D)^d=\mathcal{L}(D-Q_r)=0$, if $\deg(D)-\deg(Q_r)<0$.
\end{itemize}

Now we get  \[ \overline{\mathbb{D}}(\overline{\cN_t}\setminus \cN)=\underset{D\in\mathcal{D}}{\limsup}\frac{\mid (\overline{\cN_t}\setminus\cN)\cap\mathcal{L}(D)^d\mid}{\mid\mathcal{L}(D)^d\mid}\leq\underset{D\in\mathcal{D}}{\limsup}\left\vert\bigcup_{\substack{r>t \\ \deg(Q_r)\leq\deg(D)}}Q_r^d\cap\mathcal{L}(D)^d\right\vert q^{-d\ell(D)}  \] 
\begin{equation}\label{cappu}
= \underset{D\in\mathcal{D}}{\limsup}\left\vert\bigcup_{\substack{r>t \\ \deg(Q_r)\leq\deg(D)}}\mathcal{L}(D-Q_r)^d\right\vert q^{-d\ell(D)}\leq\underset{D\in\mathcal{D}}{\limsup}\sum_{\substack{r>t \\ \deg(Q_r)\leq\deg(D)}}\frac{q^{d\ell(D-Q_r)}}{q^{d\ell(D)}}.
\end{equation}
Observe now that if $\deg(D-Q_r)\geq 0$ we have that $\ell(D-Q_r)\leq\deg(D-Q_r)+1$ \citep[Eq. 1.21]{bib:stichtenoth2009algebraic} and also that $\ell(D)\geq\deg(D)+1-g$ by Riemann's theorem.\\
Hence we have that (\ref{cappu}) is less or equal than \[ \underset{D\in\mathcal{D}}{\limsup}\sum_{\substack{r>t \\ \deg(Q_r)\leq\deg(D)}}\frac{q^{d(1+\deg(D)-\deg(Q_r))}}{q^{d(\deg(D)+1-g)}}\leq \sum_{r>t}q^{d(g-\deg(Q_r))}=q^{dg}\sum_{r>t}q^{-d\deg(Q_r)}. \]
We now notice that $\sum_{r>t}q^{-d\deg(Q_r)}$ is the tail of a subseries of the Zeta function, which is absolutely convergent for $d>1$. Letting $t$ going to infinity the tail converges to $0$, thus $\underset{t\rightarrow\infty}{\lim}\overline{\mathbb{D}}(\overline{\cN_t}\setminus \cN)=0$. We are now able to apply Lemma \ref{magic_lemma} with $n=d$, $A_t=\overline{\cN_t}$ and $A=\cN$\\
\[ \mathbb{D}(\cN)=\underset{t\rightarrow\infty}{\lim}\mathbb{D}(\overline{\cN_t})=\underset{t\rightarrow\infty}{\lim}\prod_{i=1}^t \left(1-\frac{q^{\deg(Q_i)}-1}{q^{(d+1)\deg(Q_i)}}\right)=\prod_{Q\in \cS}\left(1-\frac{q^{\deg(Q)}-1}{q^{(d+1)\deg(Q)}}\right).\]
We conclude by taking the complement \[ \mathbb{D}(\cE)=1-\mathbb{D}(\cN)=1-\prod_{Q\in \cS}\left(1-\frac{q^{\deg(Q)}-1}{q^{(d+1)\deg(Q)}}\right).\]
\end{proof}
\section{Non-Monic Eisenstein Polynomials}\label{nonmonic}

In this section we compute the density of Eisenstein polynomials applying the same strategy of section \ref{monic}. For this let $\overline{\cN_t}^{+}=\bigcap_{i=1}^t \cN_i^{+}$.

\begin{proposition} The density of $\overline{\cN_t}^{+}$ is
 \[\mathbb{D}(\overline{\cN_t}^{+})=\prod_{i=1}^t \left(1-\frac{(q^{\deg(Q_i)}-1)^2}{q^{(d+2)\deg(Q_i)}}\right).\]
\end{proposition}

\begin{proof}
Consider a divisor $D\in\mathcal{D}$. With the same reasoning of the monic case one can show that $\mathcal{L}(D)^{d+1}$ maps surjectively onto $\left(H / ( \widehat Q_1\cdots \widehat Q_t)\right)^{d+1}$ when the degree of $D$ is large enough.

Let $\psi_i:\left(H / ( \widehat Q_1\cdots \widehat Q_t)\right)^{d+1}\longrightarrow \left(H / \widehat Q_i\right)^{d+1}$ as before.
The situation is now the following: \[ \mathcal{L}(D)^{d+1}\overset{\tilde{\phi}}{\twoheadrightarrow}\left(H / ( \widehat Q_1\cdots\widehat Q_t)\right)^{d+1}\overset{\psi}{\rightarrow}\prod_{i=1}^t \left(H / \widehat Q_i\right)^{d+1},\] where $\psi =(\psi_1,\dots ,\psi_t)$.\\
Analogously to the case of monic polynomials we note that we can verify that $f\in H$ is not Eisenstein with respect to $Q_i$ by looking at the reduction modulo $\widehat Q_i$. Hence $f\in \overline{\cN_t}^{+}\cap\mathcal{L}(D)^{d+1} $ if and only if $\psi_i\circ\tilde{\phi}(f)\notin \left( (Q_i/\widehat Q_i)\setminus\{0\}\right)\times\left( Q_i/\widehat Q_i \right)^{d-1}\times \left(  (H/\widehat Q_i)\setminus (Q_i/\widehat Q_i) \right)=:E_i^{+}$ for all $i\in \{1,\dots ,t\}$.\\
Hence we get \[\mid\overline{\cN_t}^{+}\cap\mathcal{L}(D)^{d+1}\mid= \mid\ker(\tilde{\phi})\mid\cdot\prod_{i=1}^t\mid\left(H/\widehat Q_i\right)^{d+1}\setminus E_i^{+}\mid\] \[=q^{(d+1)\left(\deg(D)-2\sum_{i=1}^t\deg(Q_i)+1-g\right)}\cdot\prod_{i=1}^t\mid\left(H/\widehat Q_i\right)^{d+1}\setminus E_i^{+}\mid,\] where

\begin{align*}
  \mid\left(H/\widehat Q_i\right)^{d+1}\setminus E_i^{+}\mid &= q^{2(d+1)\deg(Q_i)}-\mid \left( (Q_i/\widehat Q_i)\setminus\{0\}\right)\times\left( Q_i/\widehat Q_i \right)^{d-1}\times \left(  (H/\widehat Q_i)\setminus (Q_i/\widehat Q_i) \right)\mid 
\\&=q^{2(d+1)\deg(Q_i)}-\left( \left( q^{\deg(Q_i)}-1 \right)q^{(d-1)\deg(Q_i)}\left( q^{2\deg(Q_i)}-q^{\deg(Q_i)}\right) \right) 
\\&=q^{2(d+1)\deg(Q_i)} \left(1-\frac{q^{2\deg(Q_i)}-2q^{\deg(Q_i)}+1}{q^{(d+2)\deg(Q_i)}}  \right)  \\&= q^{2(d+1)\deg(Q_i)} \left(1-\frac{\left( q^{\deg(Q_i)}-1\right)^2}{q^{(d+2)\deg(Q_i)}}  \right).
\end{align*}

Therefore for $D$ of degree large enough
\begin{align*}  \frac{\mid \overline{\cN_t}^{+}\cap\mathcal{L}(D)^{d+1}\mid}{\mid\mathcal{L}(D)^{d+1}\mid} &=\frac{q^{(d+1)\left(\deg(D)-2\sum_{i=1}^t\deg(Q_i)+1-g\right)}}{q^{(d+1)\left(\deg(D)+1-g\right)}}\cdot\prod_{i=1}^t q^{2(d+1)\deg(Q_i)} \left(1-\frac{\left( q^{\deg(Q_i)}-1\right)^2}{q^{(d+2)\deg(Q_i)}}  \right) \\&=\prod_{i=1}^t  \left(1-\frac{\left( q^{\deg(Q_i)}-1\right)^2}{q^{(d+2)\deg(Q_i)}}  \right). 
\end{align*}
Hence \[\mathbb{D}(\overline{\cN_t}^{+})=\underset{D\in\mathcal{D}}{\lim}\frac{\mid \overline{\cN_t}^{+}\cap\mathcal{L}(D)^{d+1}\mid}{\mid\mathcal{L}(D)^{d+1}\mid}=\prod_{i=1}^t  \left(1-\frac{\left( q^{\deg(Q_i)}-1\right)^2}{q^{(d+2)\deg(Q_i)}}  \right).\]
\end{proof}
\begin{theorem}
The density of the set of Eisenstein polynomials with coefficients in $H$ is 
\[\mathbb{D}(\cE^+)=1-\prod_{Q\in \cS}\left(1-\frac{(q^{\deg(Q)}-1)^2}{q^{(d+2)\deg(Q)}}\right).\]
\end{theorem}
\begin{proof}
Again by lemma \ref{magic_lemma} we have to show  that $ \lim_{t\rightarrow\infty}\overline{\mathbb{D}}(\overline{\cN_t}^{+}\setminus \cN^{+})=0$.\\
Observe that $\cE_r^{+}\cap\mathcal{L}(D)^{d+1}\subseteq Q_r^d\times \mathcal{L}(D)$.
We get 
\begin{align*}
 \overline{\mathbb{D}}(\overline{\cN_t}^{+}\setminus \cN^{+}) &=\underset{D\in\mathcal{D}}{\limsup}\frac{\mid (\overline{\cN_t}^{+}\setminus\cN^{+})\cap\mathcal{L}(D)^{d+1}\mid}{\mid\mathcal{L}(D)^{d+1}\mid} \\&\leq\underset{D\in\mathcal{D}}{\limsup}\left\vert\bigcup_{\substack{r>t \\ \deg(Q_r)\leq\deg(D)}}\cE_r^{+}\cap\mathcal{L}(D)^{d+1}\right\vert q^{-(d+1)\ell(D)}   
 \\&\leq\underset{D\in\mathcal{D}}{\limsup}\left\vert\bigcup_{\substack{r>t \\ \deg(Q_r)\leq\deg(D)}}\left( Q_r^d\times\mathcal{L}(D)\right)\cap\mathcal{L}(D)^{d+1}\right\vert q^{-(d+1)\ell(D)} 
\\&\leq\underset{D\in\mathcal{D}}{\limsup}\sum_{\substack{r>t \\ \deg(Q_r)\leq\deg(D)}}\frac{\mid\left( Q_r^d\times\mathcal{L}(D)\right)\cap\mathcal{L}(D)^{d+1}\mid}{q^{(d+1)\ell(D)}} \\&=\underset{D\in\mathcal{D}}{\limsup}\sum_{\substack{r>t \\ \deg(Q_r)\leq\deg(D)}}\frac{\mid Q_r\cap\mathcal{L}(D)\mid ^d\mid\mathcal{L}(D)\mid}{q^{(d+1)\ell(D)}}  
\\&=\underset{D\in\mathcal{D}}{\limsup}\sum_{\substack{r>t \\ \deg(Q_r)\leq\deg(D)}}\frac{\mid Q_r\cap\mathcal{L}(D)\mid ^d}{q^{d\ell(D)}}
\end{align*}
which is equation (\ref{cappu}).
Hence for $t$ going to infinity we obtain $\overline{\mathbb{D}}(\overline{\cN_t}^{+}\setminus \cN^{+})=0$.

We now apply Lemma \ref{magic_lemma} with $n=d+1$, $A_t=\overline{\cN_t}^{+}$ and $A=\cN^{+}$ obtaining
\[ \mathbb{D}(\cN^{+})=\underset{t\rightarrow\infty}{\lim}\mathbb{D}(\overline{\cN_t}^{+})=\underset{t\rightarrow\infty}{\lim}\prod_{i=1}^t  \left(1-\frac{\left( q^{\deg(Q_i)}-1\right)^2}{q^{(d+2)\deg(Q_i)}}  \right)=\prod_{Q\in \cS}  \left(1-\frac{\left( q^{\deg(Q)}-1\right)^2}{q^{(d+2)\deg(Q)}}  \right).\]  

We now take the complement
\[  \mathbb{D}(\cE^{+})=1-\mathbb{D}(\cN^{+})=1-\prod_{Q\in \cS}  \left(1-\frac{\left( q^{\deg(Q)}-1\right)^2}{q^{(d+2)\deg(Q)}}  \right).\]
\end{proof}

\bibliographystyle{plainnat}
\bibliography{biblio}{}

\begin{thebibliography}{7}
\providecommand{\natexlab}[1]{#1}
\providecommand{\url}[1]{\texttt{#1}}
\expandafter\ifx\csname urlstyle\endcsname\relax
  \providecommand{\doi}[1]{doi: #1}\else
  \providecommand{\doi}{doi: \begingroup \urlstyle{rm}\Url}\fi

\bibitem[Dubickas(2003)]{dubickas2003polynomials}
Art{\=u}ras Dubickas.
\newblock Polynomials irreducible by eisenstein's criterion.
\newblock \emph{Applicable Algebra in Engineering, Communication and
  Computing}, 14\penalty0 (2):\penalty0 127--132, 2003.

\bibitem[Guo and Yang(2013)]{bib:guo2013probability}
X.~Guo and G.~Yang.
\newblock The probability of rectangular unimodular matrices over
  $\mathbb{F}_q[x]$.
\newblock \emph{Linear Algebra and its Applications}, 438\penalty0
  (6):\penalty0 2675--2682, 2013.

\bibitem[Heyman and Shparlinski(2013)]{heyman2013number}
Randell Heyman and Igor~E Shparlinski.
\newblock On the number of eisenstein polynomials of bounded height.
\newblock \emph{Applicable Algebra in Engineering, Communication and
  Computing}, 24\penalty0 (2):\penalty0 149--156, 2013.

\bibitem[Kelley(1955)]{bib:kelley1955general}
J.~L. Kelley.
\newblock \emph{General topology}.
\newblock New York: Van Nostrand, 1955.

\bibitem[Micheli and Schnyder(2014)]{micheli2014density}
Giacomo Micheli and Reto Schnyder.
\newblock On the density of coprime m-tuples over holomorphy rings.
\newblock \emph{To appear in International Journal of Number Theory}, 2014.
\newblock URL \url{arXiv:1411.6876}.

\bibitem[Stichtenoth(2009)]{bib:stichtenoth2009algebraic}
H.~Stichtenoth.
\newblock \emph{Algebraic function fields and codes}, volume 254.
\newblock Springer, 2009.

\bibitem[Sugita and Takanobu(2007)]{bib:sugita2007probability}
H.~Sugita and S.~Takanobu.
\newblock The probability of two $\mathbb{F}_q[x]$-polynomials to be coprime.
\newblock \emph{Probability and number theory, Advanced Studies in Pure
  Mathematics}, 49:\penalty0 455--478, 2007.

\end{thebibliography}

\end{document}